\documentclass[secnum,leqno]{rims-bessatsu}
\usepackage{amssymb, amsmath}
\usepackage{graphics}
\usepackage[dvips]{graphicx}
\usepackage{eepic}
\usepackage{epic}
\newtheorem{thm}{Theorem}[section]
\newtheorem{crl}[thm]{Corollary}

\newtheorem{lmm}[thm]{Lemma}

\newtheorem{prp}[thm]{Proposition}

\theoremstyle{definition}
\newtheorem{dfn}[thm]{Definition}
\theoremstyle{remark}
\newcommand{\rank}{\mathrm{rank}}
\renewcommand{\Re}{\mathrm{Re}}
\renewcommand{\Im}{\mathrm{Im}}
\newcommand{\grad}{\mathrm{grad}}

\renewcommand{\span}{\mathrm{span}}
\newcommand{\Sing}{\mathrm{Sing}}

\newcommand{\Hess}{\mathrm{Hess}}
\newcommand{\C}{\mathbb{C}}
\newcommand{\R}{\mathbb{R}}

\title{The singularities of the maps associated with Milnor fibrations}
\author{\textsc{Daiki Sumida}\footnote{Graduate School of Mathematics, Kyushu University, 744, Motooka, Nishi-ku, Fukuoka
819-0395, JAPAN.\newline e-mail: \texttt{d-sumida@math.kyushu-u.ac.jp}}}
\AuthorHead{Sumida Daiki}         
\VolumeNo{x}           
\YearNo{201x}           
\PagesNo{000--000}      
\communication{Received April 20, 201x.
Revised September 11, 201x.}      
\begin{document}
%

\maketitle

\begin{abstract}      
Milnor fibrations have been studied since 1960's.
 In this paper, we study singular points of differentiable maps,
 called Milnor fibration product maps, obtained by several Milnor fibrations.
 We give a characterization of singular points of such product maps,
 and for the case of certain weighted homogeneous polynomials,
 a criterion for a fold singular point together with its index.
\end{abstract}

\section{Introduction}

 For a singular point of a complex hypersurface, the intersection of the hypersurface and the sphere
 centered at the point with sufficiently small radius, is called a link. Milnor \cite{M} has constructed a theory
 of so-called ^^ ^^ Milnor fibrations'', which are locally trivial fibrations over the unit circle, where the total space
 is the complement of the link in the sphere. This is important in singularity theory, and is used,
 for example, to describe the topology of complex hypersurface singularities.

 Milnor's fibration theorem was later extended in order to describe the topology of complete intersection singularities
 of complex algebraic varieties by Hamm \cite{H}. It provides rich and important examples of locally trivial fibrations like Milnor fibrations.
 On the other hand, it is important to investigate singular points of differentiable maps in the study of
 differentiable structures of manifolds. However, in general, it is difficult to construct explicit examples
 in which the singularities are completely determined. If we can clarify the singularities of differentiable
 maps constructed by using simple objects like polynomials by means of a straightforward technique, then
 we can expect a rich family of examples of differentiable maps for which the relations among the singular points are clear.
 In this paper, we define Milnor fibration product maps, which generalize the Milnor fibrations in a direction
 different from Hamm's, and study their singularities in detail.
 A Milnor fibration product map is a differentiable map into the torus
 $T^r = S^1 \times S^1 \times \cdots \times S^1$, $r \geq 1$, which is constructed by $r$ hypersurfaces
 in complex $n$-space passing through a common point. Its source is the complement to the union of the links
 in a small sphere, and each component function is a Milnor fibration map.

 We study singular points of Milnor fibration product maps into the tori, especially those into the $2$-dimensional torus.
 First we give a necessary and sufficient condition for a point of the source to be a singular point of the map,
 applying the techniques of Milnor. As a corollary,
 we obtain examples of Milnor fibration product maps which have no singular points. On the other hand, when the polynomials
 defining the Milnor fibration product map are a pair of weighted homogeneous polynomials
 whose weights are linearly dependent, we have a corollary which refines the above necessary
 and sufficient condition. This corollary makes it easy to determine the singular point set,
 since this singular point set turns out to be contained in a certain complex algebraic set determined by the weighted homogeneous polynomials, while the singular point set is a real algebraic set in general. In addition, for such a singular point obtained by
 using this result, we give a necessary and sufficient condition for it to be a fold point
 in a  practical and efficient form. Fold singular point is one of the generic singular points of differentiable maps
 into $2$-dimensional manifolds. The key technique to obtain a characterization of a fold point for Milnor fibration product maps, is to choose good local coordinates around
 the singular point of the Milnor fibration product map, and to see if the singular point is a non-degenerate
 critical point of one of the Milnor fibration maps restricted to the Milnor fiber of the other fibration.
 Furthermore, we determine the index of such a fold singular point when the normal space of the Milnor fiber at
 the point is a complex $1$-dimensional vector space. Using these results, we give concrete examples for which Milnor fibration
 product maps have only fold points as their singular points, and determine their indices completely.
 We also give examples of Milnor fibration product maps which have singular points such that none of them are fold points.
 In the case where the pair of polynomials that
 define the Milnor fibration product map is a pair of homogeneous polynomials in two variables,
 we show that the singular point set is a union of circles and we also study the number of components.
 Finally some of these results are extended to Milnor fibration product maps into the tori of higher dimensions.

\section{Preliminaries}
\subsection{Milnor fibration product maps}
Let $m$ and $n$ be integers with $n \geq 2$ and $1 \leq m \leq 2n-1$.
Let $f_j \in \C[z_1, z_2, \ldots, z_n]$ be polynomials with $f_j(0)=0$, $j=1, 2, \ldots, m$.
 For a sufficiently small positive real number $\varepsilon$,
 set $S_\varepsilon=\left\{z\in \C^n\,|\,\left\|z\right\|=\varepsilon\right\}$
 and $K_{f_{1}f_{2}\cdots f_m}=\left\{z\in S_\varepsilon\,|\,f_{1}(z)f_{2}(z)\cdots f_m(z)=0\right\}$.
 Let us consider the map
$$\Phi:S_\varepsilon \setminus K_{f_{1}f_{2}\cdots f_m}\to S^1\times S^1\times\cdots\times S^1=T^m$$
 defined by
\begin{eqnarray*}
\Phi(z)=\left(\frac{f_{1}(z)}{\left|f_{1}(z)\right|},
\frac{f_{2}(z)}{\left|f_{2}(z)\right|},\ldots,
\frac{f_m(z)}{\left|f_m(z)\right|}\right),
\ z\in S_\varepsilon \setminus K_{f_{1}f_{2}\cdots f_m}.
\end{eqnarray*} 
We call the $C^\infty$ map $\Phi$ the \emph{Milnor fibration product map} associated with $m$ polynomials $f_1,f_2,\ldots,f_m$.

\subsection{Notations}
For complex vectors $u, v \in \C^n$, the symbol $\langle u, v \rangle$ will
 denote the usual Hermitian inner product of $u$ and $v$, i.e.
$$\langle u, v \rangle = \sum_{j=1}^n u_j \overline{v_j},$$
where $u = (u_1, u_2, \ldots, u_n)$ and $v = (v_1, v_2, \ldots, v_n)$.
Note that if we regard $u$ and $v$ as real vectors in $\C^n = \R^{2n}$,
 then their usual Euclidean inner product is given by $\Re \langle u, v \rangle$.
For a holomorphic function $h$ defined on an open subset $U$ of $\C^n$, $\Hess_p(h)$
 will denote the $n \times n$ Hessian matrix of $h$ at $p\in U$, i.e.
$$\Hess_{p}(h) = \left(\frac{\partial^2 h}{\partial z_j \partial z_k}(p) \right)_{j, k}.$$
Furthermore, for a matrix $V$, $V^T$ will denote its transposed matrix, and $\overline{V}$
 will denote its complex conjugate matrix.

\section{Singular points of Milnor fibration product maps}
Let $f$ and $g$ be polynomials in $\C\left[z_{1},z_{2},\ldots,z_{n}\right]$ with $f(0)=g(0)=0$. 
In this section, we study singular points of the Milnor fibration product map $\Phi$ associated with $f$ and $g$.

First we give a necessary and sufficient condition for a point to be a singular point of $\Phi$.
\begin{lmm}
Let $f$ and $g$ be polynomials in $\C\left[z_1,z_2,\ldots,z_n\right]$ with $f(0)=g(0)=0$. 
A point $p\in S_\varepsilon \setminus K_{fg}$ is a singular point of the Milnor fibration product map $\Phi$ associated with $f$ and $g$
 if and only if the three vectors $p,i\,\grad\log f(p)$ and $i\,\grad\log g(p)$ are linearly dependent over $\R$.
\label{31}
\end{lmm}

\begin{proof}
For $p\in S_\varepsilon \setminus K_{fg}$,
 we have
 the differential
 $d\Phi_{p}:T_{p}(S_\varepsilon \setminus K_{fg})\to T_{\Phi(p)}T^2$ satisfies
\begin{eqnarray*}
d\Phi_{p}(v)=\left(i\Re\left\langle v,i\,\grad\log f(p)\right\rangle
\frac{f(p)}{\left|f(p)\right|},i\Re\left\langle v,i\,\grad
\log g(p)\right\rangle\frac{g(p)}{\left|g(p)\right|}\right)
\end{eqnarray*}
for $v\in T_{p}(S_\varepsilon \setminus K_{fg})$. Thus,
 $p$ is a singular point of $\Phi$ if and only if there exists
 $(\beta,\gamma)\in \R^2\setminus\left\{(0,0)\right\}$
 such that for every $v\in T_{p}(S_\varepsilon \setminus K_{fg})$,
\begin{eqnarray*}
& &\Re\left\langle v,\beta i\,\grad\log f(p)+\gamma i\,\grad\log g(p)\right\rangle \\
&=&\beta\Re\left\langle v,i\,\grad\log f(p)\right\rangle+\gamma\Re\left\langle
 v,i\,\grad\log g(p)\right\rangle=0
\end{eqnarray*}
holds.\ Note that two vectors in $\C^n$ are orthogonal to each other with respect
 to the usual Euclidean inner product if and only if the real part of their
 Hermitian inner product vanishes. Consequently,
 for every $v\in\C^n$, we have $v\in T_p(S_\varepsilon \setminus K_{fg})$ if and only if
 $\Re\left\langle v,z\right\rangle=0$ holds.
Therefore, if $p$ is a singular point, then $p$ and $\beta i\,\grad\log f(p)+\gamma i\,\grad\log g(p)$
 are linearly dependent over $\R$.
 Thus, it follows that $p,\ i\,\grad\log f(p)$ and $i\,\grad\log g(p)$ 
are linearly dependent over $\R$.

Conversely, if $p,i\,\grad\log f(p)$
 and $i\,\grad\log g(p)$ are linearly dependent over $\R$,
 then we have $\rank\,d\Phi_{p}<2$. Therefore, $p$ is a singular point of the Milnor fibration product map.
\end{proof}

The following lemma is due to Milnor \cite{M}.

\begin{lmm}
 Let $f$ be a polynomial in $\C\left[z_{1},z_{2},\ldots,z_{n}\right]$ with $f(0)=0$, where $n\geq 1$.
 A point $p\in S_\varepsilon \setminus K_{f}$ is a singular point of
 the $C^\infty$ map $\varphi_f : S_\varepsilon \setminus K_f \to S^1$ defined by
 $\varphi_f(z) = f(z)/|f(z)|,\ z \in S_\varepsilon \setminus K_f$,
 if and only if the two vectors $p$ and $i\,\grad\log f(p)$ are
 linearly dependent over $\R$.
Furthermore, for $\varepsilon>0$ sufficiently small, $\varphi_f$ has no singular point.
\label{32}
\end{lmm}

We have the following proposition to Lemmas \ref{31} and \ref{32}.

\begin{prp}
Let $f$ be a polynomial in $\C\left[z_{1},z_{2},\ldots,z_m\right]$
and $g$ be a polynomial in $\C\left[z_{m+1},z_{m+2},\ldots,z_{n}\right]$ with $f(0)=g(0)=0$.
We regard $f$ and $g$ as polynomials in $\C\left[z_{1},z_{2},\ldots,z_{n}\right]$.
 Then, for $\varepsilon>0$ sufficiently small,
 the Milnor fibration product map $\Phi: S_\varepsilon \setminus K_{fg} \to T^2$
 associated with $f$ and $g$ has no singular point.
\label{33}
\end{prp}

\begin{proof}
Let $\hat{f}$ (or $\hat{g}$) be the polynomial $f$ (resp.\ $g$) considered as
 an element of $\C[z_1, z_2, \ldots, z_m]$ (resp.\ $\C[z_{m+1}, z_{m+2}, \ldots, z_n]$).
 Let $\varepsilon$ be a small positive real number such that for every positive real number
 $\delta\leq\varepsilon$,
 $\varphi_{\hat{f}}=\hat{f}/|\hat{f}|:S_{\delta}^{2m-1}\setminus K_{\hat{f}}\to S^1$
 and $\varphi_{\hat{g}}=\hat{g}/\left|\hat{g}\right|:S_{\delta}^{2(n-m)-1}\setminus K_{\hat{g}}\to S^1$ 
have no singular points, where $S_{\delta}^{2m-1}\subset \C^m$ and
 $S_{\delta}^{2(n-m)-1}\subset \C^{n-m}$ are the spheres of radius $\delta$
centered at the origins. Let $p=(p_1,p_2,\ldots,p_n)$ be a point in $S_\varepsilon \setminus K_{fg}$
 satisfying, for $(\alpha,\beta,\gamma)\in \R^{3}$,
$$\alpha p+\beta i\,\grad\log f(p)+\gamma i\,\grad\log g(p)=0.$$
Then, the points  $p_{(1)}=(p_1,p_2,\ldots,p_m)$ and $p_{(2)}=(p_{m+1},p_{m+2},\ldots,p_n)$ satisfy the equations

\begin{displaymath}
\left\{
\begin{array}{l}
\displaystyle{
\alpha p_{(1)}+\frac{\beta i}{\,\overline{\hat{f}(p_{(1)})}\,}
\left(\overline{\frac{\partial\hat{f}(p_{(1)})}{\partial z_{1}}},
\overline{\frac{\partial\hat{f}(p_{(1)})}{\partial z_{2}}},
\ldots,\overline{\frac{\partial\hat{f}(p_{(1)})}{\partial z_m}}\right)=0}
,\\
\displaystyle{
\alpha p_{(2)}+\frac{\gamma i}{\,\overline{\hat{g}(p_{(2)})}\,}
\left(\overline{\frac{\partial\hat{g}(p_{(2)})}{\partial z_{m+1}}},
\overline{\frac{\partial\hat{g}(p_{(2)})}{\partial z_{m+2}}},
\ldots,\overline{\frac{\partial\hat{g}(p_{(2)})}{\partial z_{n}}}\right)=0}.
\end{array}
\right.
\end{displaymath}
Note that the points $p_{(1)}$ and $p_{(2)}$ are not the origins,
 since we have $\hat{f}(p_{(1)})\neq0$ and $\hat{g}(p_{(2)})\neq0$.
 Thus, we have $0<\varepsilon_{(1)}=|p_{(1)}|,\varepsilon_{(2)}=|p_{(2)}|\leq\varepsilon$.
 From the assumption on $\varepsilon$, it follows that $p_{(1)}$ is a regular point of
 $\varphi_{\hat{f}}:S_{\varepsilon_{(1)}}^{2m-1}\setminus K_{\hat{f}}\to S^{1}$
 and $p_{(2)}$ is a regular point of $\varphi_{\hat{g}}:S_{\varepsilon_{(2)}}^{2(n-m)-1}\setminus K_{\hat{g}}\to S^{1}$
. From Lemma \ref{32}, it follows that $(\alpha,\beta,\gamma)=(0,0,0)$, i.e.
 $p,i\,\grad\log f(p)$ and $i\,\grad\log g(p)$
 are linearly independent over $\R$. Therefore, $p$
 is a regular point of the Milnor fibration product map $\Phi$ associated with $f$ and $g$ by Lemma \ref{31}.
\end{proof}

 We can extend Lemma \ref{31} and Proposition \ref{33}
 to the case of Milnor fibration product maps associated with $m$ polynomials
 for general $m \geq 2$,
 which will be proved in Section $5$. 
 
Let us define weighted homogeneous polynomials.

\begin{dfn}
A polynomial $f\in \C[z_1, z_2, \ldots, z_n]$ is called a \emph{weighted homogeneous polynomial}
 if there exists a sequence $w_{f}=(w_1,w_2,\ldots,w_n)$ of positive rational numbers,
 called \emph{weights}, such that
\begin{eqnarray*}
\displaystyle \sum_{j=1}^{n}\frac{z_{j}}{w_{j}}\frac{\partial f}{\partial z_{j}}=f
\end{eqnarray*}
holds.
\end{dfn}

For a weighted homogeneous polynomial $f$,
 the weights of $f$ are
 not uniquely determined in general.
 For example, $f=z_1z_2$ is a weighted homogeneous polynomial in $\C[z_1,z_2]$
 with weights $w_f=(a,a/(a-1))$ for every rational number $a>1$.

The following characterizations of weighted homogeneous polynomials are known \cite{M}.

\begin{lmm}
For a polynomial $f$ in $\C\left[z_{1},z_{2},\ldots,z_{n}\right]$,
 the following conditions are equivalent.
\begin{itemize}
\item The polynomial $f$ is a weighted homogeneous with weights $w_{f}=(w_1,w_2,\ldots,w_n)$.
\item We have $f(z_{1}e^{c/w_1},z_{2}e^{c/w_2},\ldots,z_{n}e^{c/w_n})=e^{c}f(z_1,z_2,\ldots,z_n)$ 
for every $c\in \C$.
\item For every non-zero monomial $Cz_{1}^{b_1}z_{2}^{b_2}\cdots z_{n}^{b_n}$ of $f$, we have
$$\displaystyle \sum_{j=1}^{n}\frac{b_{j}}{w_{j}}=1.$$
\end{itemize}
\end{lmm}

 For certain weighted homogeneous polynomials $f$ and $g$, we have the following corollary to Lemma \textup{3.1}.

\begin{crl}
Let $f$ and $g$ be weighted homogeneous polynomials
 with weights $w_{f}$ and $w_{g}$, respectively.
 Suppose that there exists a positive rational number $s$ such that $w_{g}=sw_{f}$.
Then a point $p\in S_\varepsilon \setminus K_{fg}$ is a singular point of 
the Milnor fibration product map $\Phi: S_\varepsilon \setminus K_{fg} \to T^2$ associated with $f$ and $g$ if and only if
$$
\frac{\partial f}{\partial z_{j}}(p)\frac{\partial g}{\partial z_{k}}(p)
-\frac{\partial f}{\partial z_{k}}(p)\frac{\partial g}{\partial z_{j}}(p)=0
$$
holds for all $j,k=1,2,\ldots,n$.
\end{crl}

\begin{proof}
Set $z=(z_{1},z_{2},\ldots,z_{n})$. Let $w(z)$ be the vector defined by
$$w(z)=\left(\frac{z_1}{w_1},\frac{z_2}{w_2},\ldots,\frac{z_n}{w_n}\right)$$
for the weights $w_f=(w_1,w_2,\ldots,w_n)$.
 From the definition of weighted homogeneous polynomials, it follows that
\begin{eqnarray*}
\left\langle w(z),\grad\log f\right\rangle
= \sum_{j=1}^n\frac{z_{j}}{w_{j}}\frac{1}{f}\frac{\partial f}{\partial z_{j}}
= \frac{1}{f}\sum_{j=1}^n\frac{z_{j}}{w_{j}}\frac{\partial f}{\partial z_{j}}=1,\\
\left\langle w(z),\grad\log g\right\rangle
= \sum_{j=1}^n\frac{z_{j}}{w_{j}}\frac{1}{g}\frac{\partial g}{\partial z_{j}}
= \frac{s}{g}\sum_{j=1}^n\frac{z_{j}}{sw_{j}}\frac{\partial g}{\partial z_{j}}=s.
\end{eqnarray*}
For a singular point $p=(p_{1},p_{2},\ldots,p_{n})$ of $\Phi$,
 there exists $(\alpha,\beta,\gamma)\in \R^{3}\setminus\left\{(0,0,0)\right\}$
 such that $\alpha p+\beta i\,\grad\log f(p)+\gamma i\,\grad\log g(p)=0$ by Lemma \textup{3.1}.
 From this fact, it follows that
\begin{equation}
\alpha\left\langle w(p),p\right\rangle-i(\beta+\gamma s)=\left\langle w(p),\alpha p
+\beta i\,\grad\log f(p)+\gamma i\,\grad\log g(p)\right\rangle=0.
\end{equation}
We see that $\langle w(p),p \rangle$ is a positive real number, since
$$\left\langle w(p),p\right\rangle=\sum_{j=1}^{n}\frac{p_{j}}{w_{j}}
\overline{p_{j}}=\sum_{j=1}^{n}\frac{1}{w_{j}}\left|p_{j}\right|^{2}.$$
From this fact and formula (3.1), it follows that $\alpha=0$.
 Therefore, $\grad\log f(p)$ and $\grad\log g(p)$ are linearly dependent over $\C$,
 and $p$ satisfies the required condition.

Conversely, if $p$ satisfies the condition of Corollary \textup{3.6},
 i.e.\ if there exists $(\beta,\gamma)\in \C^2 \setminus\left\{(0,0)\right\}$ such that
\begin{equation}
\beta\,\grad\log f(p)+\gamma\,\grad\log g(p)=0
\end{equation}
holds, then it follows that $\beta=-\gamma s\neq 0$, since
$$
\overline{\beta}+\overline{\gamma}s=\left\langle w(p),\beta\,\grad\log f(p)
+\gamma\,\grad\log g(p)\right\rangle=0.
$$
From equation (3.2) and the equation $\overline{\gamma}(\beta,\gamma)=(-|\gamma|^2s,|\gamma|^2)$,
 it follows that $p$ satisfies the equation
$$-|\gamma|^2 s i\,\grad\log f(p)+|\gamma|^2 i\,\grad\log g(p)=0.$$
We see that the three vectors $p,i\,\grad\log f(p),i\,\grad\log g(p)$ are
 linearly dependent over $\R$, since $(0,-|\gamma|^2s,|\gamma|^2)$ is a non-zero real vector.
 Therefore, $p$ is a singular point of the Milnor fibration product map $\Phi$
 associated with $f$ and $g$ by Lemma \textup{3.1}.
\end{proof}

We have the following theorem, which gives a necessary and sufficient condition
 for a singular point of a Milnor fibration product map to be a fold point.

\begin{thm}
Let $f$ and $g \in \C\left[z_1,z_2,\ldots,z_n\right]$ be weighted homogeneous polynomials in
  with weights $w_f$ and $w_g$,
 respectively, such that $w_g=sw_f$, where $s$ is a positive rational number.
Let $p$ be a singular point of the Milnor fibration product map $\Phi:S_\varepsilon \setminus K_{fg}\to T^{2}$
associated with $f$ and $g$.
 Let $V=(r_1^T,r_2^T,\ldots,r_{2n-2}^T)$ be an $ n \times (2n-2)$ complex matrix, where $r_1,r_2,\ldots,r_{2n-2}$ constitute
 a real basis of the real $(2n-2)$-dimensional vector space
\begin{eqnarray*}
\left\{v\in \C^{n}\,|\,\Re\left\langle v,p\right\rangle
=\Re\left\langle v,i\grad\log f(p)\right\rangle=0\right\}.
\end{eqnarray*}
 Then $p$ is a fold point of $\Phi$ if and only if
\begin{eqnarray*}
\det \Re ( V^T \Hess_{p}(-i(\log g -s \log f)) V)\neq 0
\end{eqnarray*}
holds.
\end{thm}

\begin{proof}
Let $w_f = (w_1, w_2, \ldots, w_n)$ be the weights of $f$.
For $t \in \R$, let us define the diffeomorphism $h_t: \C^n \to \C^n$ by
$$h_t(z_1, z_2, \ldots, z_n) =
\left(z_1 \exp\left(\frac{2 \pi i t}{w_1}\right),
z_2 \exp\left(\frac{2 \pi i t}{w_2}\right), \ldots,
z_n \exp\left(\frac{2 \pi i t}{w_n}\right)\right),$$
for $z = (z_1, z_2, \ldots, z_n) \in \C^n$.
Note that $||h_t(z)|| = ||z||$ and that
\begin{eqnarray*}
f(h_t(z)) & = & f(z) \exp{(2 \pi i t)}, \\
g(h_t(z)) & = & g(z) \exp{(2 \pi i st)}.
\end{eqnarray*}
Therefore, the diffeomorphism $h_t$ preserves $S_\varepsilon \setminus K_f$,
$S_\varepsilon \setminus K_g$ and $S_\varepsilon \setminus K_{fg}$,
and in particular,
it induces diffeomorphisms between Milnor fibers for $f$.

For $p \in S_\varepsilon \setminus K_{fg}$,
we choose local coordinates $(x_1, x_2, \ldots, x_{2n-1})$
around $p$ and $(y_1, y_2)$ around $\Phi(p) \in T^2$ as
follows. For the Milnor fiber $F_f$ for $f$ containing $p$,
let $(x_1, x_2, \ldots, x_{2n-2})$ be local coordinates
for $F_f$ around $p$. We extend these coordinates
to those for $S_\varepsilon \setminus K_{fg}$ around $p$
in such a way that the first $2n-2$ coordinates
$x_1, x_2, \ldots, x_{2n-2}$ are invariant under the
diffeomorphism $h_t$ for $|t|$ sufficiently small, and that
the $(2n-1)$-th coordinate $x_{2n-1}$ satisfies
$$\varphi_f(p) \exp{(2 \pi i x_{2n-1})} 
= \varphi_f(x_1, x_2, \ldots, x_{2n-1}).$$
Set $\Phi(p) = (\exp{(2 \pi i \alpha_0)}, \exp{(2 \pi i \beta_0)})$
for $\alpha_0, \beta_0 \in \R$.
We can choose local coordinates $(y_1, y_2)$ for $T^2$ around
$\Phi(p)$ so that 
$$\left(\exp{(2 \pi i (\alpha + \alpha_0))}, 
\exp{(2 \pi i (\beta + \beta_0))}\right) \in T^2$$ 
corresponds to
$$(y_1, y_2) = (\alpha, \beta - s \alpha),$$
or in other words,
$$(\alpha, \beta) = (y_1, y_2 + s y_1).$$

Then, with respect to the above local coordinates, we have
$$y\circ\Phi (0, 0, \ldots, 0, x_{2n-1})
= (x_{2n-1}, 0),$$ and $y\circ\Phi (x_1, x_2, \ldots, x_{2n-2}, 0)$
is of the form $(0, \rho(x_1, x_2, \ldots, x_{2n-2}))$, where
$\rho$ is a smooth function defined on a neighborhood
of $p$ in $F_f$.
Consequently, the differential $d\Phi_{p}$ has rank one,
and by virtue of the well-known
characterization of a fold point (for example,
see \cite{GG}), we see that
$p$ is a fold point of $\Phi$
if and only if $p \in F_f$ is a non-degenerate critical
point of $\rho$.


Note that we have 
\begin{eqnarray*}
\rho(x_1, x_2, \ldots, x_{2n-2}) & = & 
y_2 \circ \Phi(z(x_1, x_2, \ldots, x_{2n-2}, 0)) \\
& = & \Re(-i(\log g - s \log f))(x_1, x_2, \ldots, x_{2n-2}, 0).
\end{eqnarray*}
Thus, for $k = 1, 2, \ldots, 2n-2$, we have
\begin{eqnarray*}
\frac{\partial \rho}{\partial x_k}
& = & \Re \left(\sum_{a=1}^n \frac{\partial}{\partial z_a}
(-i(\log g - s \log f)) \frac{\partial z_a}{\partial x_k}\right) \\
& = & \Re \left\langle v_k, i \, 
\mathrm{grad}(\log g - s \log f) \right\rangle,
\end{eqnarray*}
where 
$$
v_k = \left( \frac{\partial z_1}{\partial x_k},
\frac{\partial z_2}{\partial x_k}, \ldots,
\frac{\partial z_n}{\partial x_k}\right).
$$
(Note that the vectors $v_1(p), v_2(p), \ldots, v_{2n-2}(p)
\in \C^n$
constitute a basis of the real $(2n-2)$-dimensional vector
space mentioned in the theorem.)
Therefore, for $m = 1, 2, \ldots, 2n-2$, we have
\begin{eqnarray*}
\frac{\partial^2 \rho}{\partial x_k \partial x_m}(p)
& = & \Re \left\langle \frac{\partial v_k}{\partial x_m}(p), i \, 
\mathrm{grad}(\log g - s \log f)(p) \right\rangle \\
& & \quad +
\Re \left\langle v_k(p), i\, \sum_{b=1}^n
\frac{\partial \,\mathrm{grad}(\log g - s \log f)}
{\partial \overline{z_b}}(p) \overline{\frac{\partial z_b}{\partial x_m}(p)} 
\right \rangle\\
& = & \Re \left(\sum_{a=1}^n \sum_{b=1}^n 
\frac{\partial z_a}{\partial x_k}(p) 
\frac{\partial^2 (-i(\log g - s \log f))}{\partial z_a \partial z_b}(p)
\frac{\partial z_b}{\partial x_m}(p)\right),
\end{eqnarray*}
since we have
$$i \,\mathrm{grad}(\log g - s \log f)(p) = 0$$
by the proof of Corollary~\textup{3.6}.
Thus, we have the desired conclusion.
\end{proof}

\begin{crl}
\label{a}
Let $f$ and $g \in \C\left[z_1,z_2,\ldots,z_n\right]$ be weighted homogeneous polynomials
 with weights $w_f$ and $w_g$,
 respectively, such that $w_g=sw_f$, where $s$ is a positive rational number.
Let $p$ be a singular point of the Milnor fibration product map $\Phi:S_\varepsilon \setminus K_{fg}\to T^{2}$
associated with $f$ and $g$ such that $p$ and
 $i\,\grad\log f(p)$ are linearly dependent over $\C$.
 Let $W=(r_1^T,r_2^T,\ldots,r_{n-1}^T)$ be an $n\times(n-1)$ complex matrix, where $r_1,r_2,\ldots,r_{n-1}$ constitute
 a complex basis of the $(n-1)$-dimensional complex vector space
$\left\{v\in \C^{n}|\left\langle v,p\right\rangle=0\right\}$.
 Then $p$ is a fold point of $\Phi$ if and only if
\begin{eqnarray*}
\det(W^{T}\Hess_{p}(-i(\log g-s\log f))W)\neq 0
\end{eqnarray*}
holds.
\end{crl}

\begin{proof}
We see that $r_1,r_2,\ldots,r_{n-1},
ir_1,ir_2,\ldots,ir_{n-1}\in \C^n$
 constitute a real basis of
\begin{eqnarray*}
\left\{v\in \C^{n}\,|\,\Re \left\langle v,p\right\rangle
=\Re\left\langle v,i\grad\log f(p)\right\rangle=0\right\}.
\end{eqnarray*}
 Set $W=(r_1^T,r_2^T,\cdots,r_{n-1}^T)$ and $V=(W,iW)$.
 Putting $H=\Hess_{p}(-i(\log g-s\log f))$, we have
\begin{eqnarray*}
\det\Re(V^{T}HV)
&=& \det\Re\left((W,iW)^{T}H(W,iW)\right) \\
&=& \det
\begin{pmatrix}
\Re(W^{T}HW) & -\Im(W^{T}HW) \\
-\Im(W^{T}HW) & -\Re(W^{T}HW)
\end{pmatrix}
\\
&=& (-1)^{n-1}|\det(W^{T}HW)|^2.
\end{eqnarray*}
Therefore, by Theorem ~\textup{3.7}, we have the desired conclusion.
\end{proof}

We can determine the index of the fold point in the situation of Corollary ~\textup{3.8}
 as follows.
\begin{prp}
Let $f$ and $g\in\C[z_1,z_2,\ldots,z_n]$ be
 weighted homogeneous polynomials with weights $w_f$ and $w_g$,
 respectively, such that $w_g=sw_f$, where $s$ is a positive rational number.
 Suppose that $p \in S_\varepsilon \setminus K_{fg}$ is a fold point of
 the Milnor fibration product map associated with $f$ and $g$.
 Then, its index is equal to $n-1$, if $p$ and $i\,\grad\,\log f(p)$
 are linearly dependent over $\C$.
\end{prp}
\begin{proof}
We use the same notations as in the proof of Corollary ~\textup{3.8}.
 The eigenvalues of $\Re(V^{T}HV)$ are real numbers,
 since $\Re(V^{T}HV)$ is a real symmetric matrix.
 For a positive integer $m$, let $I_m$ denote the $m\times m$ identity matrix.
 Let us calculate the characteristic polynomial of $\Re(V^{T}HV)$:
\begin{eqnarray*}
& &\det\left(tI_{2n-2}-\Re(V^{T}HV)\right)\\
&=&\det
\begin{pmatrix}
tI_{n-1}-\Re(W^{T}HW) &\Im(W^{T}HW)\\
\Im(W^{T}HW) &tI_{n-1}+\Re(W^{T}HW)
\end{pmatrix}.
\end{eqnarray*}
For every non-zero complex number $t$, and $(n-1)\times (n-1)$ matrices $A,B$ and $I=I_{n-1}$,
\begin{eqnarray*}
&&
\begin{pmatrix}
I & -(1/2it)(tI-A+iB)\\
0 & I
\end{pmatrix}
\begin{pmatrix}
I & 0\\
iI & I
\end{pmatrix}
\begin{pmatrix}
tI-A & B\\
B & tI+A
\end{pmatrix}
\begin{pmatrix}
I & 0\\
iI & I
\end{pmatrix}
\\&=&
\begin{pmatrix}
0 & B-(1/2it)(tI-A+iB)(tI+A+iB)\\
2itI & tI+A+iB
\end{pmatrix}
\end{eqnarray*}
holds. Therefore,
\begin{eqnarray*}
\det
\begin{pmatrix}
tI-A & B\\
B & tI+A
\end{pmatrix}
&=&\det((tI-A+iB)(tI+A+iB)-2itB)\\
&=&\det(t^{2}I-(A-iB)(A+iB))
\end{eqnarray*}
holds. Consequently, for every non-zero complex number $t$, we have
\begin{eqnarray*}
& &\det\left(tI_{2n-2}-\Re(V^{T}HV)\right)\\
&=&\det(t^{2}I_{n-1}-\overline{W^{T}HW}W^{T}HW).
\end{eqnarray*}
By the proof of Corollary ~\textup{3.8}, this equality holds also for $t=0$. 
Thus, for every eigenvalue of $\overline{W^T H W}W^T H W$,
 its square root must be eigenvalues of $\Re(V^{T}HV)$,
 which are real numbers, since $\Re(V^{T}HV)$ is a real symmetric matrix.
 When $p$ is a fold point, the eigenvalues $\pm \lambda_1, \pm \lambda_2, \ldots, \pm \lambda_{n-1}$
 of $\Re(V^{T}HV)$ are non-zero real numbers. Therefore,
 the index of the fold point $p$ is equal to $n-1$.
\end{proof}

\section{Examples of singular points}
In this section, we give explicit examples of singular points of Milnor fibration product maps into the $2$-dimensional torus. We also clarify their fold points.
\begin{prp}
Let $f=\sum_{j=1}^{n} c_j z_j^m$ and $g=\sum_{k=1}^{n} d_k z_k^m$
 be homogeneous polynomials in $\C\left[z_1,z_2,\ldots,z_n\right]$, where
 $c_{j}\neq 0, d_{k}\neq 0$ for all $j$ and $k$, and $A_{j,k}=c_{j}d_{k}-c_{k}d_{j}\neq 0$ for $j\neq k$, and $m\geq 2$.
 Let $\Phi:S_\varepsilon \setminus K_{fg}\to T^2$ be the Milnor fibration product map associated with $f$ and $g$.
 Then, for $m=2$, $\Phi$ is a fold map. For $m>2$, $\Phi$ has no fold point, although it has singular points.
\end{prp}
\begin{proof}
The polynomials $f$ and $g$ are weighted homogeneous with weights $w_{f}=w_{g}=(m,m,\ldots,m)$.
 Therefore, it satisfies the condition of Corollary ~\textup{3.6}. Let $\Sing(\Phi)$
 be the singular point set of $\Phi$. Then we see that
\begin{eqnarray*}
\Sing(\Phi)=\{\varepsilon e^{i\theta} e_u \,|\, \theta \in \R, 1 \leq u \leq n\}
\end{eqnarray*}
by Corollary ~\textup{3.6}, where
\begin{eqnarray*}
e_1=(1 , 0 , \ldots , 0),e_2=(0,1,\ldots,0),\ldots,e_n=(0,0,\ldots,1)\in\C^{n}.
\end{eqnarray*}
Note that $\Sing(\Phi)$ consists of $n$ circles.
 Then, every singular point $p$ of $\Phi$ satisfies
\begin{eqnarray*}
i\,\grad\displaystyle \log f(p)=\frac{mi}{\varepsilon^2}p,
\end{eqnarray*}
i.e., the hypothesis of Corollary ~\textup{3.8} is satisfied.
 For $p=\varepsilon e^{i\theta} e_u\in \Sing(\Phi)$, we see that
\begin{eqnarray*}
\span_{\C} \{\varepsilon e^{i\theta}e_j\,|\,j\neq u,\ 1\leq j \leq n\}
\end{eqnarray*}
is the complex orthogonal complement to 
the complex vector space spanned by $p$ for $u=1,2,\ldots,n$.
 Let $W$ be the $n\times (n-1)$ complex matrix given by
\begin{eqnarray*}
W=\varepsilon e^{i\theta}(e_1,\cdots,e_{u-1},e_{u+1},\cdots,e_n).
\end{eqnarray*}
 Furthermore, by direct computation, we have
\begin{displaymath}\Hess_{p}(-i(\log g-s\log f))=
\begin{cases}
\displaystyle{\frac{2ie^{-2i\theta}}{\varepsilon^2 c_{u}d_{u}}(\delta_{j,k}A_{j,u})_{j,k=1}^n}, & m=2, \\
\hspace{15mm}0, & m>2,
\end{cases}
\end{displaymath}
for $p=\varepsilon e^{i\theta}e_{u}$, where
\begin{eqnarray*}
\delta_{j,k}=
\begin{cases}
1, & j=k, \\
0, & j\neq k.
\end{cases}
\end{eqnarray*}
Note that for $m=2$, this is a diagonal matrix with non-zero diagonals
 except for the $(u, u)$-component. Then we have
\begin{displaymath}W^{T}\Hess_{p}(-i(\log g-s\log f))W=
\begin{cases}
\displaystyle{\frac{2i}{c_{u}d_{u}}}(\delta_{j,k}A_{j,u})_{j,k\neq u}, & m=2,\\
\hspace{15mm}0, & m>2,
\end{cases}
\end{displaymath}
where $(\delta_{j,k}A_{j,u})_{j,k\neq u}$ is the $(n-1)\times (n-1)$ diagonal matrix
 which has non-zero diagonal entries $A_{1,u},\ldots,A_{u-1,u},A_{u+1,u},\ldots,A_{n,u}$. 
Therefore, by Corollary ~\textup{3.8}, none of the singular point of $\Phi$
 is a fold point for $m>2$. Furthermore, $\Phi$ is a fold map for $m=2$
 by our assumption on the coefficients $c_{j}$ and $d_{k}$.
\end{proof}
\begin{prp}
Let $f=z_1^m+z_2^m$ and $g=z_1z_2$ be polynomials in $\C[z_1,z_2]$.
 The Milnor fibration product map
 $\Phi:S_\varepsilon \setminus K_{fg}\to T^2$ associated with $f$ and $g$
 is a fold map for all $m \geq 2$.
 Furthermore, the indices of the fold points are all equal to $1$.
\end{prp}

\begin{proof}
The two polynomials $f$ and $g$ are weighted homogeneous with weights
 $w_{f}=(m,m)$ and $w_{g}=(2,2)$, and we have $w_g=sw_f$ for $s=2/m$. Then we have
\begin{eqnarray*}
\Sing(\Phi)=\left\{\frac{\varepsilon e^{i\theta}}{\sqrt2}(1,\omega)\,\Big|\,
\omega^m=1,\ \theta \in \R\right\}
\end{eqnarray*}
by Corollary ~\textup{3.6}. For the singular point $p=\displaystyle 2^{-1/2}\varepsilon e^{i\theta}(1,\omega)$ of $\Phi$, we have
\begin{eqnarray*}
i\,\grad\log f(p)=\frac{mi}{\varepsilon^2}p,
\end{eqnarray*}
and hence, every singular point $p$ of $\Phi$ satisfies the hypothesis
 of Corollary ~\textup{3.8}. Furthermore, we can put
\begin{eqnarray*}
W=\frac{\varepsilon e^{i\theta}}{\sqrt2}
\begin{pmatrix}
1\\
-\omega
\end{pmatrix},
\end{eqnarray*}
since
\begin{eqnarray*}
\left\langle \frac{\varepsilon e^{i\theta}}{\sqrt2}
\begin{pmatrix}
1,\omega
\end{pmatrix}
,\frac{\varepsilon e^{i\theta}}{\sqrt2}
\begin{pmatrix}
1,-\omega
\end{pmatrix}
\right\rangle
=\frac{\varepsilon^2}{2}(1-\omega \overline{\omega})=0
\end{eqnarray*}
holds. By direct computation, we have
\begin{eqnarray*}
\Hess_{p}\left(-i(\log g-s\log f)\right)
=\frac{mie^{-2i\theta}}{\varepsilon^2\omega^2}\left(
\begin{array}{cc}
\omega^2 & -\omega\\
-\omega & 1
\end{array}\right),
\end{eqnarray*}
and hence,
$$
W^{T}\Hess_{p}\left(-i(\log g-s\log f)\right)W=2mi.
$$
Therefore, $\Phi$ is a fold map by Corollary 3.8.
 Furthermore, the indices of the fold points of $\Phi$ are all equal to $n-1$ by Proposition 3.9. 
\end{proof}
We have a result about the number of connected components of the singular point set $\Sing(\Phi)$.
\begin{prp} 
Let $f$ and $g$ be two homogeneous polynomials in $\C[z_1,z_2]$.
 Then, the singular point set of the Milnor fibration product map $\Phi:S_\varepsilon \setminus K_{fg}\to T^2$ associated with $f$ and $g$
 consists of at most $\deg f+\deg g-2$ circle components or the whole
 $S_\varepsilon \setminus K_{fg}$.
\end{prp}
\begin{proof}
The polynomial
$$R(z_1,z_2)=\frac{\partial f}{\partial z_1}\frac{\partial g}{\partial z_2}-\frac{\partial f}{\partial z_2}\frac{\partial g}{\partial z_1}$$
is a homogeneous polynomial with degree $d$, where $d=\deg f+\deg g-2$, or the zero polynomial.
 Every $2$-variable homogeneous polynomial can be written as a product of some homogeneous polynomials of degree one; i.e.\ there exist $\alpha_1,\ldots,\alpha_{d},\beta_1,\ldots,\beta_{d}\in \C$
 such that
$$
R(z_1,z_2)=(\alpha_1z_1-\beta_1z_2)\cdots(\alpha_{d}z_1-\beta_{d}z_2).
$$
If $R(z_1,z_2)$ is the zero polynomial, then we see that every point of
 $S_\varepsilon \setminus K_{fg}$ is a singular point of $\Phi$ by Corollary 3.6.
 If $R(z_1,z_2)$ is a non-zero polynomial, then by Corollary 3.6 $\Sing(\Phi)$ is contained in the set
\begin{eqnarray*}
L &=& \left\{(z_1,z_2)\in S_\varepsilon \,\Big|\,R(z_1,z_2)=0\right\} \\
&=& \left\{\frac{\varepsilon e^{i\theta}}{\left\|(\alpha_{j},\beta_{j})\right\|}(\beta_{j},\alpha_{j})\,\Big|\,j=1,2,\ldots,d,\,\theta\in\R\right\}.
\end{eqnarray*}
We have that $p\in L$ is an element of $S_\varepsilon \setminus K_{fg}$ if and only if $f(p)g(p)\neq 0$ holds. 
For every real number $t$, a point $e^{it}p$ is an element of $S_\varepsilon \setminus K_{fg}$ if and only if
 $p$ is an element of $S_\varepsilon \setminus K_{fg}$, since 
$$f(e^{it}z)=e^{it\,\deg f}f(z),\ g(e^{it}z)=e^{it\,\deg g}g(z)$$
 hold. Therefore, $\Sing(\Phi)$ consists of
 at most $d$ circle components.
\end{proof}

\section{Milnor fibration product maps into higher dimensional tori}
In this section, we extend some of the results obtained in Section ~\textup{3} to Milnor fibration product maps into the $m$-dimensional torus for general $m \geq 2$.

\begin{lmm}
Let $f_j, j=1,2,\ldots,m$, be polynomials in $\C[z_1,z_2,\ldots,z_n]$ with $f_{j}(0)=0$ for all $j$. A point $p\in S_\varepsilon \setminus K_{f_1 f_2 \cdots f_m}$ is a singular point
 of the Milnor fibration product map $\Phi:S_\varepsilon \setminus K_{f_1 f_2 \cdots f_m}\to T^m$
 if and only if the $m+1$ vectors
$$p,i\,\grad\log f_1(p),i\,\grad\log f_2(p),\ldots, i\,\grad\log f_m(p)$$
 are linearly dependent over $\R$.
\end{lmm}
\begin{proof}
As in the proof of Lemma \ref{31}, for $p\in S_\varepsilon \setminus K_{f}$, the differential $d\Phi_{p}:T_{p}(S_\varepsilon \setminus K_{f})\to T_{\Phi(p)}T^m$ satisfies
\begin{center}
$
\displaystyle
d\Phi_{p}(v)=\left(i\Re\langle v,i\,\grad\log f_1
(p)\rangle\frac{f_1(p)}{|f_1(p)|},
\ldots,i\Re\langle v,i\,\grad\log f_m(p)\rangle
\frac{f_m(p)}{|f_m(p)|}\right),
$
\end{center}
for $v$ in $T_{p}(S_\varepsilon \setminus K_{f})$.
Therefore, $p$ is a singular point of $\Phi$ if and only if there exists
 $(\beta_1,\beta_2,\ldots,\beta_m)\in \R^m\setminus\left\{(0,0,\ldots,0)\right\}$
 such that for every $v\in T_{p}(S_\varepsilon \setminus K_{f})$,
\begin{eqnarray*}
\Re\left\langle v,\sum_{j=1}^m \beta_{j}i\,\grad\log f_{j}(p)\right\rangle
=\sum_{j=1}^m\beta_{j}\Re\left\langle v,i\,\grad\log f_{j}(p)\right\rangle=0
\end{eqnarray*}
holds. In this case, $p$ and $\Sigma_{j=1}^m \beta_{j}i\,\grad\log f_{j}(p)$ are linearly dependent over $\R$. Consequently $p,i\,\grad\log f_1(p),i\,\grad\log f_2(p),\ldots,i\,\grad\log f_m(p)$ are linearly dependent over $\R$. Conversely if $p,i\,\grad\log f_1(p),i\,\grad\log f_2(p),\ldots,i\,\grad\log f_m(p)$ 
are linearly dependent over $\R$, then we have $\rank\,d\Phi_{p}<m$. Therefore, $p$ is a singular point of $\Phi$.
\end{proof}

\begin{prp}
For $j\in\{1,2,\ldots,m\}$,
let $f_{j}\in\C[z_{k_{j-1}+1},z_{k_{j-1}+2},\ldots,
z_{k_{j}}]$ be a polynomial with $f_{j}(0)=0$, 
 where $0=k_0<k_1<\cdots<k_m=n$.
 We regard $f_1,f_2,\ldots,f_m$ as polynomials in
 $\C\left[z_1,z_2,\ldots,z_n\right]$.
 For $\varepsilon>0$ sufficiently small,
 the Milnor fibration product map
 $\Phi:S_\varepsilon \setminus K_{f_1f_2\cdots f_m}\to T^m$
 associated with $f_1,f_2,\ldots,f_m$ has no singular point.
\end{prp}

\begin{proof}
Let $\hat{f_j}$ be the polynomial $f_j,j=1,2,\ldots,m$, considered as
 an element of $\C[z_{k_{j-1}+1}, z_{k_{j-1}+2}, \ldots, z_{k_{j}}]$.
 Let $\varepsilon$ be a small positive real number such that for every positive real number
 $\delta\leq\varepsilon$,
\begin{eqnarray*}
\varphi_{\hat{f_j}}=\frac{\hat{f_j}}{|\hat{f_j}|}:S_{\delta}^{2(k_j-k_{j-1})-1}\setminus K_{\hat{f_j}}\to S^1, j=1,2,\ldots,m
\end{eqnarray*}
have no singular points, where $S_{\delta}^{2(k_j-k_{j-1})-1}\subset \C^{k_j-k_{j-1}}$, $j=1,2,\ldots,m$, are the spheres of radius $\delta$
centered at the origins. Let $p=(p_1,p_2,\ldots,p_n)$ be a point in $S_\varepsilon \setminus K_{f_1f_2\cdots f_m}$
 satisfying, for $(\alpha,\beta_1,\beta_2,\ldots,\beta_m)\in \R^{m+1}$,
$$\alpha p+\sum_{j=1}^m\beta_j i\,\grad\log f_j(p)=0.$$
Set $P_j=(p_{k_{j-1}+1},p_{k_{j-1}+2},\ldots,p_{k_{j}}),j=1,2,\ldots,m$. Then $P_1,P_2,\ldots,P_m$ satisfy the equations
\begin{eqnarray*}
\alpha P_j+\frac{\beta_{j}i}{\ \overline{f_j(P_j)}\ }\left(\overline{\frac{\partial f_{j}(P_j)
}{\partial z_{k_{j-1}+1}}},\overline{\frac{\partial f_{j}(P_j)
}{\partial z_{k_{j-1}+2}}},\ldots,
\overline{\frac{\partial f_{j}(P_j)}{\partial z_{k_{j}}}}\right)=0.
\end{eqnarray*}
By the assumption on $\varepsilon$, $P_j$ is a regular point of
 $\varphi_{\hat{f_{j}}}$,
 and hence, by Lemma \ref{32}, we have $(\alpha,\beta_{j})=(0,0)$
 for every $j=1,2,\ldots,m$.
 Therefore, the $m+1$ vectors $p,i\,\grad\log f_1(p),\ldots,i\,\grad \log f_m(p)$ are
 linearly independent over $\R$, i.e.\ $p$ is a regular point of $\Phi$.
\end{proof}
\begin{prp}
Let $f_1,f_2,\ldots,f_m$, $2\leq m\leq 2n-1$, be weighted homogeneous polynomials in
 $\C[z_1,z_2,\ldots,z_n]$ with weights
 $w_{f_1}, w_{f_2}, \ldots, w_{f_m}$, respectively. Suppose that $w_{f_j} = s_j w_{f_1}$ for some positive integer $s_j$ for $j = 1, 2, \ldots, m$.
 If $p \in S_\varepsilon \setminus K_{f_1f_2 \cdots f_m}$ is a singular point of
 the Milnor fibration product map
 $\Phi:S_\varepsilon \setminus K_{f_1f_2\cdots f_m}\to T^m$
 associated with $f_1,f_2,\ldots,f_m$, then we have that the $m$-vectors
\begin{eqnarray*}
i\,\grad\log f_1(p),\ldots,i\,\grad\log f_m(p)
\end{eqnarray*}
are linearly dependent over $\C$.
\end{prp}
\begin{proof}
By Lemma ~\textup{5.1}, there exists $(\alpha,\beta_1,\beta_2,\ldots,\beta_m)\in \R^{m+1}\setminus\left\{(0,0,\ldots,0)\right\}$ such that
$$
\alpha p+\sum_{j=1}^m\beta_{j}i\,\grad\log f_{j}(p)=0.
$$
From the definition of weighted homogeneous polynomials, it follows that
\begin{equation*}
\left\langle w(z),\grad\log f_{j}\right\rangle
=\frac{1}{f_{j}}\sum_{k=1}^n\frac{z_{k}}{a_{k}}
\frac{\partial f_{j}}{\partial z_{k}}
=\frac{s_{j}}{f_{j}}\sum_{k=1}^n\frac{z_{k}}{s_{j}a_{k}}
\frac{\partial f_{j}}{\partial z_{k}}=s_{j},\ z\in \C^n\setminus f_j^{-1}(0)
\end{equation*}
for $j=1,2,\ldots,m$, where $w_{f_1}=(a_1,a_2,\ldots,a_n)$ and
\begin{eqnarray*}
w(z)=\left(\frac{z_1}{a_1},\frac{z_2}{a_2},\ldots,\frac{z_n}{a_n}\right).
\end{eqnarray*}
Therefore, we have
\begin{center}
$\alpha\left\langle w(p),p\right\rangle-i\displaystyle \sum_{j=1}^m\beta_{j}s_{j}=\left\langle w(p),
\alpha p+\sum_{j=1}^m\beta_{j}i\,\grad\log f_{j}(p)\right\rangle=0.$
\end{center}
Then we have $\alpha=0$ by comparing the real part of the both sides of this equation.
 Therefore, $i\,\grad\log f_1(p),i\,\grad\log f_2(p),\ldots,i\,\grad\log f_m(p)$
 are linearly dependent over $\R$. In particular, $i\,\grad\log f_1(p),
i\,\grad\log f_2(p),\ldots,i\,\grad\log f_m(p)$ are linearly dependent over 
$\C$.
\end{proof}


%

\end{document}